\documentclass[12pt]{article}
\usepackage{amsmath,amsthm,amssymb,latexsym,color}
\theoremstyle{plain}
\newtheorem{theor}{Theorem}
\theoremstyle{remark}
\newtheorem{rem}{Remark}
\theoremstyle{plain}

\newtheorem{cor}[theor]{Corollary}
\newtheorem{lemma}[theor]{Lemma}
\def\E{{\mathbb E}}
\def\R{{\mathbb R}}
\def\P{{\mathbb P}}
\def\N{{\mathbb N}}
\def\Z{{\mathbb Z}}

\def\Exp{{\mathbb E}}
\def\M{{\mathcal M}}

\def\Proj{{\rm P}}
\def\spn{{\rm span}}

\def\B{{\rm BM}}
\def\conv{{\rm conv}}
\def\BStat{{\mathcal B}}
\def\BadBlocks{{\mathcal I}}

\def\d{{\rm dist}}
\def\S{S}
\def\Id{{\bf I}}

\def\Event{{\mathcal E}}

\title{Minimax of an $n$-dimensional Brownian motion}
\author{Konstantin Tikhomirov \ and \ Pierre Youssef\footnote{University of Alberta, Department of Mathematical and Statistical sciences. \hspace{3cm} \texttt{e-mail: \small ktikhomi@ualberta.ca ; pyoussef@ualberta.ca}}}

\begin{document}
\maketitle

\abstract{For some absolute constants $c$, $n_0$ and any $n\geq n_0$, we show that with probability close to one the 
convex hull of the $n$-dimensional Brownian motion $\conv\{\B_n(t):\, t\in[1,2^{cn}]\}$ does not contain the origin. 
The result can be interpreted as an estimate of the minimax of the Gaussian process $\{ \langle \bar{u},\B_n(t)\rangle,\, \bar{u}\in\S^{n-1},\, t\in [1,2^{cn}]\}$. 
}
\section{Introduction}


Our paper is motivated by the following question raised by I.~Benjamini and considered by R.~Eldan in \cite{MR3161524}:

\vskip 0.2cm
{\it
Let $t_1,t_2,\dots,t_N$ be points in $[0,1]$ generated by a homogeneous Poisson point process with intensity $\alpha$. 
Estimate the value $\alpha=\alpha_0$ such that the convex hull of $\B_n(t_i),i=1,2,\dots,N$, contains the origin with probability $1/2$.} 
\vskip 0.2cm

Here, $\B_n$ is the standard Brownian motion in $\R^n$.
Eldan \cite{MR3161524} showed that $\alpha_0$ satisfies  
\begin{equation}\label{eq-eldan}
e^{c_1n/\log n}\le \alpha_0\le e^{c_2n\log n},
\end{equation}
for some universal constants $c_1$ and $c_2$. 
Related results were obtained in \cite{MR3161524} for
the standard random walk on $\Z^n$ and the spherical Brownian motion. 
The right-hand side estimate in \eqref{eq-eldan} was recently improved to
$e^{c_2n}$ by the authors \cite{kostya-pierre}. In fact, \cite{kostya-pierre} 
provides a rather general method for estimating from below 
the probability of the event $0\in \{W(t)\}$ for various types of random walks $W$ in $\R^n$.
At the same time, the question of optimizing the lower bound for $\alpha_0$ in \eqref{eq-eldan}
remained open.

The main result of this paper is the following theorem.

\begin{theor}\label{main result}
There exist universal constants $c>0$ and $n_0\in\N$ with the following property: let $n\ge n_0$ and
$\B_n(t)$ ($0\le t<\infty$)
be the Brownian motion in $\R^n$. Then
$$\P\bigl\{0\in\conv\{\B_n(t):\,t\in [1,2^{cn}]\}\bigr\}\le \frac{1}{n}.$$
\end{theor}

\begin{rem}
The bound $\frac{1}{n}$ in the above theorem can be replaced with $\frac{1}{n^L}$
for any constant $L>0$ at expense of decreasing $c$ and increasing $n_0$.
\end{rem}

As an immediate corollary of Theorem~\ref{main result}, we get 

\begin{cor}
There exist universal constants $\tilde c>0$ and $n_0\in\N$ with the following property:
Let $n\ge n_0$ and let $\B_n(t)$ ($t\in[0,\infty)$) be the standard Brownian motion in $\R^n$.
Further, let $t_1,t_2,\dots,t_N$ be points generated by the homogeneous Poisson process on $[0,1]$
of intensity $\alpha>0$, which is independent from the process $\B_n$.
If $\alpha\le \exp(\tilde cn)$ then
$$\P\bigl\{0\in\conv\{\B(t_i):\,i\le N\}\bigr\}\le \frac{1}{n}.$$
\end{cor}

In particular, we improve the left-hand side estimate in \eqref{eq-eldan} to $e^{c_1n}\le \alpha_0$ and,
together with the aforementioned result of \cite{kostya-pierre}, provide 
the optimal bounds for $\alpha_0$, up to the choice of $c_1$ and $c_2$.

\vskip 0.2cm

The main result of this paper is equivalent to the estimate 
$$
\P\bigl\{ \min_{u\in\S^{n-1}}\max_{t\in[1,2^{cn}]} \langle u, \B_n(t)\rangle <0\bigr\} \geq 1-\frac{1}{n}.
$$
We note that the minimax of certain Gaussian processes was studied in
\cite{MR800188}, \cite{MR950977} (see also \cite[Theorem~3.16]{MR1102015}). Those 
results found applications in Asymptotic Geometric Analysis (Dvoretzky's Theorem) and the theory of compressed sensing (see \cite{candes}). 

\vskip 0.2cm 

We think that it may be of interest to consider the following generalization of the question studied in this paper:
\vskip 0.2cm 
{\it
Let $X(t)$ be a centered Gaussian process in $\R^n$. Estimate the distribution of 
$$\min_{u\in\S^{n-1}}\max_{t} \langle u, X(t)\rangle$$ in terms of the covariance structure of the process $X$.} 
The corresponding question of estimating (up to a constant multiple) $\Exp \sup_t Y(t)$ for a $1$-dimensional 
Gaussian process $Y$ was solved by Fernique and Talagrand (see \cite{Talagrand} and references therein). 

\vskip 0.2cm

Let us give an informal description of the proof of the main result. We construct a random unit vector 
$\bar{n}$ in $\R^n$ such that with probability close to one 
\begin{equation}\label{eq-intro-goal}
\langle \bar{n}, \B_n(t)\rangle >0 \quad \text{for any }  t\in [1,2^{cn}].
\end{equation}
The construction procedure shall be divided into a series of steps. At the initial step, we produce 
a random vector $\bar{n}_0$ such that
$$
\langle \bar{n}_0, \B_n(2^i)\rangle >0 \quad \text{for any } i=0,1,\dots, cn.
$$
(In fact, $\bar{n}_0$ will satisfy a stronger condition). At a step $k$, $k\ge 1$, 
we ``update'' the vector $\bar{n}_{k-1}$ by adding a small ``perturbation'' in such a way that 
$$
\langle \bar{n}_{k}, \B_n(2^{j2^{-k}})\rangle >0  \quad \text{for any } j=0,1,\dots, 2^kcn.
$$
(Again $\bar{n}_k$ will in fact satisfy a stronger condition). Finally, using some standard properties of the Brownian 
bridge, we verify that 
$\bar{n}:= \bar{n}_{\ln\ln n}$ satisfies \eqref{eq-intro-goal} with a large probability.

\section{Preliminaries}

In this section we introduce some notation and state several auxiliary results that will be used within the proof.

By $\{e_i\}_{i=1}^n$ we denote the standard unit basis in $\R^n$,
by $\|\cdot\|$ ~--- the canonical Euclidean norm and by $\langle \cdot,\cdot\rangle$ ~---
the corresponding inner product.
For $N\ge n$ and an $N\times n$ matrix $A$, let $s_{\max}(A)$ and $s_{\min}(A)$ be its largest and smallest singular values, respectively, i.e.\
$s_{\max}(A)=\|A\|$ (the operator norm of $A$) and $s_{\min}(A)=\inf\limits_{y\in \S^{n-1}}\|Ay\|$.
For a finite set $I$, let $|I|$ be its cardinality.
By $c,c_1,\tilde c$, etc. we denote universal constants.
To avoid difficult to read formulas, we do not use any notation for truncation of a real number to the nearest integer.
For example, the product $cn$ in the next section is always treated as an integer,
as well as several other quantities depending on $n$.

Let $(\Omega,\Sigma,\P)$ be the probability space.
Throughout the text, $\gamma$ denotes the standard Gaussian variable.
The following estimate is well known (see, for example, \cite[Lemma~VII.1.2]{Feller}):
\begin{equation}\label{feller ineq}
\P\{\gamma\ge\tau\}=\frac{1}{\sqrt{2\pi}}\int\limits_\tau^\infty\exp(-t^2/2)\,dt<\frac{1}{\sqrt{2\pi}\tau}\exp(-\tau^2/2),\;\;\tau>0.
\end{equation}
Let $n\ge m$ and let $G$ be the standard $n\times m$ Gaussian matrix. Then for any $t\ge 0$
\begin{equation}\label{gaussian matrix ineq}
\P\bigl\{\sqrt{n}-\sqrt{m}-t\le s_{\min}(G)\le s_{\max}(G)\le \sqrt{n}+\sqrt{m}+t\bigr\}\ge 1-2\exp(-t^2/2)
\end{equation}
(see, for example, \cite[Corollary~5.35]{V}).

The proof of the next Lemma is straightforward, so we omit it.
\begin{lemma}\label{brownian-segment-split}
Let $\B_n(t)$ ($0\le t>\infty$) be the standard Brownian motion in $\R^n$ and let $0<a<b$.
Fix any $s\in(a,b)$ and set
$$
w(s):= \frac{b-s}{b-a}\B_n(a)+ \frac{s-a}{b-a}\B_n(b);\;\;u(s):= \B_n(s)-w(s).
$$ 
Then
\begin{enumerate}
\item $u(s)\sim \mathcal{N}\bigl(0,\frac{(b-s)(s-a)}{b-a}\Id_n\bigr)$.
\item The random vector $u(s)$ is independent from $(\B_n(t))$, $t\in(0,a]\cup[b,\infty)$.
\end{enumerate}
\end{lemma}

\begin{lemma}\label{normal-vector-lem}
Let $d,m \in\mathbb{N}$ be such that $m\leq d/2$. Let $X_1, X_2,\ldots, X_m$ be independent standard Gaussian 
vectors in $\mathbb{R}^d$. Then for any $b\in\S^{m-1}$, there exists a random unit vector $\bar{u}\in\R^d$ 
such that 
$$
\P \Big\{ \langle \bar{u},X_i\rangle \ge c_{\ref{normal-vector-lem}}\sqrt{d}
\vert b_i\vert,\ \text{ for all } i=1,2,\ldots ,m\Big\}\ge 1-\exp(-c_{\ref{normal-vector-lem}}d),
$$
where $c_{\ref{normal-vector-lem}}$ is a universal constant and $b_i$'s are the coordinates of $b$.
Moreover, $\bar u$ can be defined as a Borel function of $X_i$'s and $b$.
\end{lemma}
\begin{proof}
Without loss of generality, we can assume that $b_i\neq 0$ for any $i\le m$
and that $X_i$'s are linearly independent on the entire probability space. 
Denote by $E$ the affine subspace spanned by $\{\vert b_i\vert^{-1}X_i\}_{i\le m}$.
Define $\bar{u}$ as the unique unit vector in $\spn\{X_1,\dots,X_m\}$ such that $\bar{u}$ is orthogonal to 
$E$ and for any $i\le m$ we have
$$\langle \bar{u}, \vert b_i\vert^{-1}X_i \rangle = \d(0, E),$$
where $\d(0, E)$ stands for the distance from the origin to $E$. Then we have
\begin{equation}\label{eq-dist-E}
\sum_{i\le m} \langle \bar{u}, X_i\rangle^2= \sum_{i\le m} \bigl\langle \bar{u}, \frac{X_i}{\vert b_i\vert}\bigr\rangle^2 \vert b_i\vert^2 
= \sum_{i\le m} \d(0, E)^2\cdot \vert b_i\vert^2= \d(0, E)^2.
\end{equation}
Let $G$ be the $m\times d$ standard Gaussian matrix with rows $X_i$, $i=1,2,\ldots, m$. 
Using the definition of $\bar{u}$ together with \eqref{eq-dist-E}, we obtain for any $\tau>0$:
\begin{align*}
\P \Big\{ \langle \bar{u},X_i\rangle \ge \tau\sqrt{d} \vert b_i\vert \text{ for all } i=1,2,\dots ,m\Big\}
&= \P \Big\{ d(0,E)\ge \tau\sqrt{d}\Big\}\\
&= \P \Big\{ \sqrt{\sum_{i\le m} \langle \bar{u}, X_i\rangle^2}\ge \tau\sqrt{d}\Big\}\\
&=\P\Big\{\Vert G\bar{u}\Vert\ge \tau\sqrt{d}\Big\}\\
&\ge \P\Big\{ s_{\min}(G)\ge \tau\sqrt{d}\Big\}.
\end{align*}
The proof is finished by choosing a sufficiently small $c_{\ref{normal-vector-lem}}:=\tau$
and applying \eqref{gaussian matrix ineq}.
\end{proof}

\begin{lemma}\label{norm of noncentered gaussian}
Let $q\in\N$ and $r\in\R$ with $e\le r\le \sqrt{\ln q}$, and let
$\gamma_1, \gamma_2,\ldots, \gamma_q$ be independent standard Gaussian variables. 
Define a random vector $b=(b_1,b_2,\dots,b_q)\in\R^q$
by $b_i:=\max(0, \gamma_i - r)$, $i\le q$. Then 
$$
\P\Big\{ \Vert b\Vert \le 4\sqrt{q}\exp(-r^2/8)\Big\} \ge 1-\exp(-2\sqrt{q}).
$$ 
\end{lemma}
\begin{proof}
Let $\lambda \in (0,1/2)$. We have 
$$
\E e^{\lambda \Vert b\Vert^2} = \prod_{i=1}^q \E e^{\lambda {b_i}^2}
= \left( 1+\int_{1}^\infty \P\{ e^{\lambda {b_1}^2}\ge \tau\} d\tau\right)^q.
$$
Next, using \eqref{feller ineq}, we get
\begin{align*}
\int_{1}^\infty \P\{ e^{\lambda b_1^2}\ge \tau\} d\tau 
&\le (r-1)\P\{\gamma_1>r\} + \int_{r}^{\infty} \P\{e^{\lambda b_1^2}\ge \tau\} d\tau\\
&\le e^{-r^2/2} +  \int_{r}^{\infty} \P\Bigl\{\gamma_1\ge \sqrt{\frac{\ln\tau}{\lambda}}\Bigr\} d\tau\\
&\le e^{-r^2/2}+ \int_{r}^{\infty} \tau^{-\frac{1}{2\lambda}}d\tau\\
&=e^{-r^2/2}+ \frac{r^{1-\frac{1}{2\lambda}}}{\frac{1}{2\lambda}-1}.
\end{align*}
Now, take $\lambda=\bigl(2+\frac{r^2}{\ln r}\bigr)^{-1}$ so that $\frac{1}{2\lambda} -1= \frac{r^2}{2\ln r}$. 
After replacing $\lambda$ with its value, we deduce that 
\begin{equation}\label{eq-laplace}
\E e^{\lambda \Vert b\Vert^2} \le \bigl( 1+2e^{-r^2/2}\bigr)^q\le \exp(2qe^{-r^2/2}).
\end{equation}
Using Markov's inequality together with \eqref{eq-laplace}, we obtain
$$
\P\{\lambda \Vert b\Vert^2\ge 4q e^{-r^2/2}\} \le \exp(-2qe^{-r^2/2})\le \exp(-2\sqrt{q}),
$$
where the last inequality holds since $r\le \sqrt{\ln q}$. To finish the proof, it remains to note that
$$
\frac{4q e^{-r^2/2}}{\lambda} \le 8qr^2e^{-r^2/2}\le 16q e^{-r^2/4}.
$$

\end{proof}



\section{The proof}\label{the proof section}

Throughout the section, we assume that $c>0$ and $n_0\in\N$ are appropriately chosen constants (with $c$ sufficiently
small and $n_0$ sufficiently large) and $n\ge n_0$ is fixed. The precise conditions on $c$ and $n_0$ can be recovered
from the proof, however, we prefer to avoid these technical details.
To prove our main result, we shall construct a random unit vector $\bar n\in\R^n$ such that
\begin{equation}\label{condition on barn}
\langle \bar n,\B_n(t)\rangle > 0\;\quad \text{for any } t\in[1,2^{cn}]
\end{equation}
with probability close to one. 

 Our construction shall be iterative; in fact, we shall produce a sequence
of random vectors $\bar n_k$, $k=0,1,\dots,M$ (with $M=\log_2 \ln n$), where each $\bar n_k$
satisfies $\langle \bar n_k,\B_n(t)\rangle >0$ for certain discrete subset of $[1,2^{cn}]$ with a high probability
(the precise condition shall be given later). The size of those discrete subsets shall grow with $k$ in such a way
that the vector $\bar n:=\bar n_{M}$ shall possess the required property \eqref{condition on barn} with probability
close to one.

Given any $0<k\le M$, the vector $\bar n_k$ shall be a ``small perturbation'' of the vector $\bar n_{k-1}$.
The operation of constructing $\bar n_k$ will be referred to as {\it the $k$-th step} of the construction.
We must admit that the construction is rather technical. In fact, each step itself shall be divided
into a sequence of {\it substeps}. To make the exposition of the proof as clear as possible, we
won't provide all the details at once but instead introduce them sequentially.

Let $M'=\frac{1}{4}\log_2 \ln n$. We split $\R^n$ into $(M+1)\times M'$ 
coordinate subspaces. Precisely, we write 
$$
\R^n:=\prod_{k=0}^{M}\prod_{\ell=1}^{M'}\R^{J_\ell^k},
$$
where $J_\ell^k$ are pairwise disjoint subsets of $\{1,2,\dots,n\}$ with
$|J_\ell^k|=c_Jn2^{-(k+\ell)/8}$ for an appropriate constant $c_J$ and
$\R^{J_\ell^k}= \spn\{e_i\}_{i\in J_\ell^k}$.  
For every $k\le M,\ell\le M'$, define $\Proj_\ell^k:\R^n\to \R^n$
as the orthogonal projection onto $\R^{J_\ell^k}$.

Let $N=cn$ and define
$$a_0=0 \text{ \ and \ }a_i:=2^{i-1},\;\;i=1,2,\dots,N+1.$$
We shall split the interval $[0,a_{N+1}]$ into ``blocks''.
{\it The zero} block is the interval $[0,1]$; for each admissible $i\ge 0$,
the $i$-th block is the interval $[a_i,a_{i+1}]$.
With the $i$-th block, we associate a sequence of sets
$I^i_k$, $k=0,1,\dots,M,$ in the following way:
for $i=0$ we have $I^i_k=\emptyset$ for all $k\ge 0$;
for $i\ge 1$, we set $I^i_0=\emptyset$ and
$$I^i_k:=\{2^{1/2^{k}}a_i,2^{2/2^k}a_i,\dots,2^{(2^k-1)/2^k}a_i\},
\;\;k=1,2,\dots,M.$$

Further, we define two functions $f,h:\N_0\times\N_0\to\R_+$
as follows:
\begin{enumerate}
\item $f$ is decreasing in both arguments;
$f(0,0)=C_f+(1-2^{-1/4})^{-2} C_f$; for each $k\ge 0$ and
$\ell>0$ we have $f(k,\ell-1)-f(k,\ell)=C_f2^{-(k+\ell)/4}$;
finally, $f(k,0)=\lim\limits_{\ell\to\infty}f(k-1,\ell)$ for all $k\ge 1$.

\item $h$ is increasing in both arguments;
$h(0,0)=0$; for each $k\ge 0$ and
$\ell>0$ we have $h(k,\ell)-h(k,\ell-1)=C_h2^{-(k+\ell)/4}$;
moreover, $h(k,0)=\lim\limits_{\ell\to\infty}h(k-1,\ell)$ for all $k\ge 1$.
\end{enumerate}

Here, $C_f=2(1-2^{-1/4})^{-2}C_h$.
Note that the definition implies $f(k,\ell)\ge C_f\ge 2h(k,\ell)$ for all admissible $k,\ell$.
The constants $c$ and $C_f$ are connected via the relation
\begin{equation}\label{relation between constants}
8c f(1,0)^2=c_J {c_{\ref{normal-vector-lem}}}^2,
\end{equation}
where $c_{\ref{normal-vector-lem}}$ is taken from Lemma~\ref{normal-vector-lem}.
Thus, the choice of $c$ will determine both $C_f$ and $C_h$. In what follows, we always assume that $c>0$ is chosen to
be very small, so that both $C_f$ and $C_h$ are very large.

Now, we can state more precisely what we mean by the $k$-th step of the construction ($k=0,1,\dots,M$). {\bf The goal
of the $k$-th step is to produce a random unit vector $\bar n_k$ with the following properties:}
\begin{align}
&\begin{aligned}{\bf 1.}\;\bar n_k\in \prod_{p=0}^{k}\prod_{\ell=1}^{M'}\R^{J^p_\ell};\end{aligned}\label{nk prop 1}\\
&\begin{aligned}&{\bf 2.}\;\bar n_k\mbox{ is measurable with respect to the $\sigma$-algebra generated by $\Proj^p_\ell \B_n(t)$,}\\
&\mbox{$0\le p\le k$, $1\le \ell\le M'$,\;
$t\in\{a_1,\dots,a_{N+1}\}\cup I^1_k\cup I^2_k\cup \dots\cup I^N_k$};
\end{aligned}\label{nk prop 2}\\
&{\bf 3.}\;\mbox{The event}\nonumber\\
&\hspace{1cm}\begin{aligned}
\Event_k=\Bigl\{&\langle \bar n_{k},\B_n(t)-\B_n(a_i)\rangle\ge -h(k+1,0)\sqrt{a_i}\mbox{ and}\Bigr.\\
\Bigl.&\langle \bar n_{k},\B_n(a_{i+1})-\B_n(a_{i})\rangle\ge f(k+1,0)\sqrt{a_{i+1}}\Bigr.\\
\Bigl.&\mbox{for all }t\in I^i_k\mbox{ and }i=0,1,\dots,N\Bigr\}
\end{aligned}\nonumber\\
&\mbox{has probability close to one}\nonumber.
\end{align}
Quantitative estimates of $\P(\Event_k)$ shall be given later.
Note that the third property, together with the definition of the functions $f$ and $h$, implies that
$$\P\bigl\{\langle \B_n(t),\bar n_k\rangle>0\mbox{ for any }t\in\{a_1,\dots,a_{N+1}\}\cup I^1_k\cup I^2_k\cup \dots\cup I^N_k\bigr\}
\ge \P(\Event_k)\approx 1.$$
Moreover, as we show later, standard estimates for the maximum of
the Brownian bridge imply \eqref{condition on barn} for $\bar n=\bar n_M$
with probability at least $\P(\Event_M)-1/n^2\approx 1$.

The vector $\bar n_0$ shall be constructed directly using Lemma~\ref{normal-vector-lem}.
For $k\ge 1$, the vectors $\bar n_k$ are obtained via an embedded iteration procedure realized as
a sequence of substeps. First, let us give a ``partial'' description of the procedure, omitting some details.

Fix $k\ge 1$ and set $\bar n_{k,0}:=\bar n_{k-1}$. We shall inductively construct random vectors $\bar n_{k,\ell}$, $1\le\ell\le M'$
using the following notion.
For each $\ell=1,2\dots,M'+1$ and every
block $i=0,1,2,\dots,N$ {\it the $i$-th block statistic} is
\begin{align*}
\BStat_i(k,\ell):=\max\Bigl(0,
&\max\limits_{t\in I^i_k}\bigl\langle \bar n_{k,\ell-1},
\frac{\B_n(a_i)-\B_n(t)}{\sqrt{a_i}}\bigr\rangle-h(k,\ell),\Bigr.\\
\Bigl.&\bigl\langle \bar n_{k,\ell-1},\frac{\B_n(a_i)-\B_n(a_{i+1})}{\sqrt{a_{i+1}}}\bigr\rangle+f(k,\ell)\Bigr).
\end{align*}
Note that for the zero block the corresponding statistic is simply
$$\max\Bigl(0,-\bigl\langle \bar n_{k,\ell-1},
\B_n(a_{1})\bigr\rangle+f(k,\ell)\Bigr).$$
The $(N+1)$-dimensional vector $\bigl(\BStat_0(k,\ell),\dots,\BStat_N(k,\ell)\bigr)$
shall be denoted by $\BStat(k,\ell)$. Let us also denote
$$\BadBlocks(k,\ell):=\bigl\{i:\,\BStat_i(k,\ell)\neq 0\bigr\}.$$
{\bf Given $\bar n_{k,\ell-1}$, the goal of the $\ell$-th substep is to construct a random unit vector $\bar n_{k,\ell}$ such that}
\begin{align}
&\begin{aligned}{\bf 1.}\;\bar n_{k,\ell}\in \prod_{(p,q)\precsim (k,\ell)}\R^{J^p_q},
\mbox{ with $(p,q)\precsim (k,\ell)$ meaning ``$p<k$ or $p=k,\,q\le \ell$''};
\end{aligned}\label{nkl prop 1}\\
&\begin{aligned}&{\bf 2.}\;\mbox{$\bar n_{k,\ell}$ is measurable with respect to the $\sigma$-algebra generated by $\Proj^p_q \B_n(t)$,}\\
&\mbox{for all $(p,q)\precsim (k,\ell)$ and $t\in\{a_1,\dots,a_{N+1}\}\cup I^1_k\cup I^2_k\cup \dots\cup I^N_k$};
\end{aligned}\label{nkl prop 2}\\
&\begin{aligned}{\bf 3.}\;\mbox{$\BStat(k,\ell+1)$ ``typically'' has a smaller Euclidean norm than $\BStat(k,\ell)$.}
\end{aligned}\nonumber
\end{align}
The third property shall be made more precise later. For now, we note that
the ``typical'' value of $\BStat(k,\ell)$ shall decrease with $\ell$ 
in such a way that, after the $M'$-th substep, the vector $\BStat(k,M'+1)$ shall be zero with probability close to one.
Juxtaposing the definition of the block statistics with that of $\Event_k$, it is easy to see that, by setting $\bar n_k:=\bar n_{k,M'}$,
we get
$$\P(\Event_k)=\P\bigl\{\BadBlocks(k,M'+1)=\emptyset\bigr\}=\P\bigl\{\BStat(k,M'+1)={\bf 0}\bigr\}\approx 1.$$

The vector $\bar n_{k,\ell}$ shall be defined as
\begin{equation}\label{def of nkl via delta}
\bar n_{k,\ell}=\frac{\bar{n}_{k,\ell-1}+\alpha_{k,\ell}\bar\Delta_{k,\ell}}{\sqrt{1+{\alpha_{k,\ell}}^2}},
\end{equation}
where $\bar\Delta_{k,\ell}$ is a random unit vector (``perturbation'')
and $\alpha_{k,\ell}=16^{-k-\ell}$.\\
{\bf The vector $\bar\Delta_{k,\ell}$ shall satisfy the following properties:}
\begin{align}
&\begin{aligned}{\bf 1.}\;\bar\Delta_{k,\ell}\in\R^{J^k_\ell};\end{aligned}\label{deltakl prop 1}\\
&\begin{aligned}&{\bf 2.}\;\mbox{$\bar\Delta_{k,\ell}$ is measurable with respect to the $\sigma$-algebra generated by $\Proj^p_q \B_n(t)$}\\
&\mbox{for all admissible $(p,q)\precsim (k,\ell)$ and $t\in\{a_1,\dots,a_{N+1}\}\cup I^1_k\cup I^2_k\cup \dots\cup I^N_k$};
\end{aligned}\label{deltakl prop 2}\\
&\begin{aligned}&{\bf 3.}\;\mbox{For any fixed subset $I\subset\{0,1,\dots,N\}$ such that $\P\{\BadBlocks(k,\ell)=I\}>0$,}\\
&\mbox{$\bar\Delta_{k,\ell}$ is {\it conditionally} independent from the collection of random vectors}\\
&\hspace{1cm}\bigl\{\Proj^k_\ell(\B_n(t)-\B_n(a_i))\;\;(t\in I^i_k\cup\{a_{i+1}\}),\;\;i\notin I\bigr\}\\
&\mbox{given the event $\{\BadBlocks(k,\ell)=I\}$.}
\end{aligned}\label{deltakl prop 3}\\
&\begin{aligned}
&{\bf 4.}\;
\mbox{The event}\\
&\hspace{2cm}\Event_{k,\ell}:=\bigl\{\BStat_i(k,\ell+1)=0\mbox{ for all }i\in\BadBlocks(k,\ell)\bigr\}\\
&\mbox{has probability close to one.}
\end{aligned}\nonumber
\end{align}
Again, we shall make the last statement more precise later.
Before that, we need to verify certain quantitative properties of the block statistics.
The next Lemma deals with the statistics for the initial substep; it is followed by a corresponding statement
for $\BStat(k,\ell)$, $\ell>1$.

\begin{lemma}[Initial substep for block statistics]\label{BStat initial substep lemma}
Fix any $1\le k\le M$ and assume that a random unit vector $\bar n_{k,0}:=\bar n_{k-1}$ satisfying
properties \eqref{nk prop 1} and \eqref{nk prop 2} has been constructed. Then
\begin{align*}
\P&\Bigl\{|\BadBlocks(k,1)|\le N\exp(-{C_h}^2 2^{k/2}/16)\mbox{ and }
\|\BStat(k,1)\|\le\frac{8\sqrt{N}}{\exp({C_h}^2 2^{k/2}/32)}\Bigr\}\\
&\ge \P(\Event_{k-1})-2\exp(-2\sqrt{N}).
\end{align*}
\end{lemma}
\begin{proof}
Let $i>0$ so that $I^i_k\neq \emptyset$.
For each $t\in I^i_k\setminus I^i_{k-1}$, let $t_L$ be the maximal number in $\{a_i\}\cup I^i_{k-1}$ strictly less
than $t$ (``left neighbour'') and, similarly, $t_R$ be the minimal number in $I^i_{k-1}\cup\{a_{i+1}\}$
strictly greater than $t$ (``right neighbour''). For every such $t$, let
$$w_t:=\frac{t_R-t}{t_R-t_L}\B_n(t_L)+\frac{t-t_L}{t_R-t_L}\B_n(t_R);\;\;u_t:=\B_n(t)-w_t.$$
It is not difficult to see that
\begin{align*}
\bigl\langle &\bar n_{k,0},\frac{\B_n(a_i)-w_t}{\sqrt{a_i}}\bigr\rangle\\
&\le \max\Bigl(\bigl\langle \bar n_{k,0},\frac{\B_n(a_i)-\B_n(t_L)}{\sqrt{a_i}}\bigr\rangle,
\bigl\langle \bar n_{k,0},\frac{\B_n(a_i)-\B_n(t_R)}{\sqrt{a_i}}\bigr\rangle\Bigr)\\
&\le \max\Bigl(0,
\max\limits_{\tau\in I^i_{k-1}}\bigl\langle \bar n_{k,0},\frac{\B_n(a_i)-\B_n(\tau)}{\sqrt{a_i}}\bigr\rangle,
\bigl\langle 2\bar n_{k,0},\frac{\B_n(a_i)-\B_n(a_{i+1})}{\sqrt{a_{i+1}}}\bigr\rangle\Bigr).
\end{align*}
Hence, the $i$-th block statistic (for $i=0,1,\dots,N$) can be (deterministically) bounded as
\begin{align*}
\BStat_i(k,1)\le\max\Bigl(0,
&\max\limits_{t\in I^i_{k-1}}\bigl\langle \bar n_{k,0},
\frac{\B_n(a_i)-\B_n(t)}{\sqrt{a_i}}\bigr\rangle-h(k,1),\Bigr.\\
&\max\limits_{t\in I^i_{k}\setminus I^i_{k-1}}\bigl\langle \bar n_{k,0},
\frac{\B_n(a_i)-w_t}{\sqrt{a_i}}\bigr\rangle-h(k,1)
+\max\limits_{t\in I^i_k\setminus I^i_{k-1}}\bigl\langle\bar n_{k,0},\frac{-u_t}{\sqrt{a_i}}\bigr\rangle,\Bigr.\\
\Bigl.&\bigl\langle \bar n_{k,0},\frac{\B_n(a_i)-\B_n(a_{i+1})}{\sqrt{a_{i+1}}}\bigr\rangle+f(k,1)\Bigr)\\
&\hspace{-1.8cm}\le\max\Bigl(0,
\max\limits_{t\in I^i_{k-1}}\bigl\langle \bar n_{k,0},
\frac{\B_n(a_i)-\B_n(t)}{\sqrt{a_i}}\bigr\rangle-h(k,0),\Bigr.\\
\Bigl.&\bigl\langle 2\bar n_{k,0},\frac{\B_n(a_i)-\B_n(a_{i+1})}{\sqrt{a_{i+1}}}\bigr\rangle+2f(k,0)\Bigr)\\
&\hspace{-1cm}+\max\Bigl(0,\max\limits_{t\in I^i_k\setminus I^i_{k-1}}\bigl\langle\bar n_{k,0},\frac{-u_t}{\sqrt{a_i}}\bigr\rangle+h(k,0)-h(k,1)\Bigr).
\end{align*}
Let us denote the first summand in the last estimate by $\xi_i$, so that
$$\BStat_i(k,1)\le \xi_i+\max\Bigl(0,\max\limits_{t\in I^i_k\setminus I^i_{k-1}}\bigl\langle\bar n_{k,0},
\frac{-u_t}{\sqrt{a_i}}\bigr\rangle+h(k,0)-h(k,1)\Bigr).$$
Note that
\begin{equation}\label{BStat lem aux 1}
\Event_{k-1}=\bigl\{\xi_i=0\mbox{ for all }i=0,1,\dots,N\bigr\}.
\end{equation}
Further, the property \eqref{nk prop 2} of the vector $\bar n_{k,0}=\bar n_{k-1}$, together with Lemma~\ref{brownian-segment-split}
and properties of the Brownian motion, imply that
the Gaussian variables $\bigl\langle\bar n_{k,0},\frac{-u_t}{\sqrt{a_i}}\bigr\rangle$
are {\it jointly} independent for $t\in I^i_k\setminus I^i_{k-1}$, $i=1,2,\dots,N$, and the variance of each one
can be estimated from above by $2^{1-k}$. Thus, the vector $\BStat(k,1)$ can be majorized coordinate-wise
by the vector
$$\bigl(\xi_i+\max\limits_{t\in I^i_k\setminus I^i_{k-1}}(0,2^{(1-k)/2}\gamma_{t}+h(k,0)-h(k,1))\bigr)_{i=0}^N,$$
where $\gamma_t$ ($t\in I^i_k\setminus I^i_{k-1}$, $i=0,1,\dots,N$) are i.i.d.\ standard Gaussians
(in fact, appropriate scalar multiples of $\bigl\langle\bar n_{k,0},\frac{-u_t}{\sqrt{a_i}}\bigr\rangle$).
Denoting by $\gamma$ the standard Gaussian variable, we get from the definition of $h$:
\begin{align*}
\P\bigl\{\max\limits_{t\in I^i_k\setminus I^i_{k-1}}(0,2^{(1-k)/2}\gamma_{t}+h(k,0)-h(k,1))>0\bigr\}
&\le 2^k\P\{\gamma>C_h2^{k/4}/2\}\\
&\le 2^k\exp(-{C_h}^2 2^{k/2}/8)\\
&\le\frac{1}{2}\exp(-{C_h}^2 2^{k/2}/16).
\end{align*}
(In the last two inequalities, we assumed that $C_h$ is sufficiently large).
Applying Hoeffding's inequality to corresponding indicators, we infer
$$|\BadBlocks(k,1)|\le |\{i:\,\xi_i\neq 0\}|+N\exp(-{C_h}^2 2^{k/2}/16)$$
with probability at least $1-\exp(-2\sqrt{N})$ (we note that, in view of the inequality $k\le M$, we have
$\frac{1}{2}\exp(-{C_h}^2 2^{k/2}/16)\ge N^{-1/4}$).
Next, it is not hard to see that the Euclidean norm of $\BStat(k,1)$ is majorized (deterministically) by the sum
$$\bigl\|(\xi_i)_{i=0}^N\bigr\|+2^{(1-k)/2}\bigl\|\bigl(\max(0,\gamma_t-C_h2^{k/4}/2)\bigr)_{t}\bigr\|,$$
with the second vector having $\sum_{i=0}^N|I^i_k\setminus I^i_{k-1}|\le 2^kN$ coordinates.
Applying Lemma~\ref{norm of noncentered gaussian} to the second vector
(note that for sufficiently large $n$ we have $C_h2^{k/4}/2\le\sqrt{\ln N}$), we get
$$\|\BStat(k,1)\|\le \bigl\|(\xi_i)_{i=0}^N\bigr\|+\frac{8 \sqrt{N}}{\exp({C_h}^2 2^{k/2}/32)}$$
with probability at least $1-\exp(-2\sqrt{N})$.
Combining the estimates with \eqref{BStat lem aux 1}, we obtain the result.
\end{proof}

\begin{lemma}[Subsequent substeps for block statistics]\label{BStat subsequent substeps lemma}
Fix any $1\le k\le M$ and $1<\ell\le M'+1$ and assume that the random unit vectors
$\bar n_{k,\ell-2}$ and $\bar \Delta_{k,\ell-1}$ satisfying properties \eqref{nkl prop 1}---\eqref{nkl prop 2}
and \eqref{deltakl prop 1}---\eqref{deltakl prop 2}---\eqref{deltakl prop 3}, respectively,
have been constructed, and $\bar n_{k,\ell-1}$ is defined according to formula \eqref{def of nkl via delta}. Then
\begin{align*}
\P&\Bigl\{|\BadBlocks(k,\ell)|\le N\exp(-{C_h}^2 2^{(k+\ell)/2})\mbox{ and }
\|\BStat(k,\ell)\|\le\frac{\sqrt{N}}{\exp({C_h}^2 2^{(k+\ell)/2})}\Bigr\}\\
&\ge \P(\Event_{k,\ell-1})-2\exp(-2\sqrt{N}).
\end{align*}
Moreover,
$$\P\bigl\{\BadBlocks(k,\ell)\neq\emptyset\bigr\}\le N\exp(-{C_h}^2/\alpha_{k,\ell-1})+1-\P(\Event_{k,\ell-1}).$$
\end{lemma}
\begin{proof}
To shorten the notation, we shall use $\alpha$ in place of $\alpha_{k,\ell-1}$ within the proof.
Using the definition of $\bar n_{k,\ell-1}$ in terms of $\bar n_{k,\ell-2}$ and $\bar \Delta_{k,\ell-1}$, we get
for every $i=0,1,\dots,N$
\begin{align*}
\BStat_i(k,\ell)=\max\Bigl(&0,
\max\limits_{t\in I^i_k}\bigl\langle \frac{\bar n_{k,\ell-2}+\alpha\bar\Delta_{k,\ell-1}}{\sqrt{1+\alpha^2}},
\frac{\B_n(a_i)-\B_n(t)}{\sqrt{a_i}}\bigr\rangle-h(k,\ell),\Bigr.\\
\Bigl.&\bigl\langle \frac{\bar n_{k,\ell-2}
+\alpha\bar\Delta_{k,\ell-1}}{\sqrt{1+\alpha^2}},\frac{\B_n(a_i)-\B_n(a_{i+1})}{\sqrt{a_{i+1}}}\bigr\rangle+f(k,\ell)\Bigr)\\
&\hspace{-1.4cm}\le\frac{\BStat_i(k,\ell-1)}{\sqrt{1+\alpha^2}}\\
+\max\Bigl(&0,\max\limits_{t\in I^i_k}\bigl\langle \alpha\bar\Delta_{k,\ell-1},
\frac{\B_n(a_i)-\B_n(t)}{\sqrt{a_i}}\bigr\rangle+h(k,\ell-1)-h(k,\ell),\Bigr.\\
\Bigl.&\bigl\langle \alpha\bar\Delta_{k,\ell-1},\frac{\B_n(a_i)-\B_n(a_{i+1})}{\sqrt{a_{i+1}}}\bigr\rangle+\sqrt{1+\alpha^2}f(k,\ell)-
f(k,\ell-1)\Bigr).
\end{align*}
Let us denote the second summand by $\eta_i$ so that
$$\BStat_i(k,\ell)\le\frac{\BStat_i(k,\ell-1)}{\sqrt{1+\alpha^2}}+\eta_i.$$
Fix for a moment any subset $I$ of $\{0,1,\dots,N\}$ such that $\P\{\BadBlocks(k,\ell-1)=I\}>0$.
A crucial observation is that, conditioned on the event $\BadBlocks(k,\ell-1)=I$, the variables
$\eta_i$, $i\notin I$, are {\it jointly} independent. This follows from properties \eqref{deltakl prop 1}, \eqref{deltakl prop 3}
of $\bar\Delta_{k,\ell-1}$ and properties of the Brownian motion.
Next, the same properties tell us that, conditioned on $\BadBlocks(k,\ell-1)=I$,
each variable $\langle\bar\Delta_{k,\ell-1},\frac{\B_n(a_i)-\B_n(t)}{\sqrt{a_i}}\rangle$, $t\in I^i_k$,
and $\langle\bar\Delta_{k,\ell-1},\frac{\B_n(a_i)-\B_n(a_{i+1})}{\sqrt{a_{i+1}}}\rangle$
have Gaussian distributions with variances at most $1$.
Further, note that, by the choice of $\alpha$ and the functions $f$ and $h$,
we have
$$\sqrt{1+\alpha^2}f(k,\ell)-f(k,\ell-1)\le h(k,\ell-1)-h(k,\ell)=-C_h 2^{(-k-\ell)/4}.$$
Thus, denoting by $\gamma$ the standard Gaussian variable, we get
\begin{align}
\P\{\eta_i>0\,|\,\BadBlocks(k,\ell-1)=I\}&\le 2^k\P\{\gamma>\alpha^{-1}C_h 2^{(-k-\ell)/4}\}\nonumber\\
&\le \frac{1}{2}\exp(-{C_h}^2\alpha^{-1}),\;\;i\in\{0,1,\dots,N\}\setminus I.\label{aux 167}
\end{align}
Hence, by Hoeffding's inequality (note that $\exp(-{C_h}^2 2^{(k+\ell)/2})> 2N^{-1/4}$):
$$\P\bigl\{|\{i\notin I:\,\eta_i>0\}|\ge N\exp(-{C_h}^2 2^{(k+\ell)/2})\,|\,\BadBlocks(k,\ell-1)=I\bigr\}\le\exp(-2\sqrt{N}).$$
Next, it is not difficult to see that for any $\tau>0$ and $i\notin I$
\begin{align*}
\P\{\eta_i^2\ge\tau\,|\,\BadBlocks(k,\ell-1)=I\}&\le 2^k\P\{\max(0,\alpha\gamma-C_h 2^{(-k-\ell)/4})^2\ge\tau\}\\
&\le 1-\exp\bigl(-2^{k+1}\P\{\max(0,\alpha\gamma-C_h 2^{(-k-\ell)/4})^2\ge\tau\}\bigr)\\
&\le 1-\P\bigl\{\max(0,\alpha\gamma-C_h 2^{(-k-\ell)/4})^2<\tau\bigr\}^{2^{k+1}}\\
&\le \P\Bigl\{\sum\limits_{j=1}^{2^{k+1}}\max(0,\alpha\gamma_j-C_h 2^{(-k-\ell)/4})^2\ge\tau\Bigr\}\\
&\le\P\Bigl\{\sum\limits_{j=1}^{2^{k+1}}\max(0,\alpha\gamma_j-4\alpha C_h 2^{(k+\ell)/4})^2\ge\tau\Bigr\},
\end{align*}
where $\gamma_j$ ($j=1,2,\dots,2^{k+1}$) are i.i.d.\ copies of $\gamma$. Hence, the conditional cdf of $\|(\eta_i)_{i\notin I}\|$
given $\BadBlocks(k,\ell-1)=I$ majorizes the cdf of
$$\alpha\bigl\|\bigl(\max(0,\gamma_j-4C_h 2^{(k+\ell)/4})\bigr)_{j=1}^{2^{k+1}N}\bigr\|$$
for i.i.d.\ standard Gaussians $\gamma_j$, $j=1,2,\dots,2^kN$. Applying Lemma~\ref{norm of noncentered gaussian}
(note that $4C_h 2^{(k+\ell)/4}\le\sqrt{\ln N}$),
we obtain
\begin{align*}
&\P\Bigl\{\|(\eta_i)_{i\notin I}\|>\frac{\sqrt{N}}{\exp({C_h}^2 2^{(k+\ell)/2})}\,\bigl|\bigr.\,\BadBlocks(k,\ell-1)=I\Bigr\}\\
&\le\P\Bigl\{\bigl\|\bigl(\max(0,\gamma_j-4C_h 2^{(k+\ell)/4})\bigr)_{j=1}^{2^{k+1}N}\bigr\|>
\frac{\alpha^{-1}\sqrt{N}}{\exp({C_h}^2 2^{(k+\ell)/2})}\,\bigl|\bigr.\,\BadBlocks(k,\ell-1)=I\Bigr\}\\
&\le\P\Bigl\{\bigl\|\bigl(\max(0,\gamma_j-4C_h 2^{(k+\ell)/4})\bigr)_{j=1}^{2^{k+1}N}\bigr\|>
\frac{4\sqrt{2^{k+1}N}}{\exp(2{C_h}^2 2^{(k+\ell)/2})}\,\bigl|\bigr.\,\BadBlocks(k,\ell-1)=I\Bigr\}\\
&\le \exp(-2\sqrt{N}).
\end{align*}
Now, clearly $\BStat_i(k,\ell-1)=0$ for all $i\notin I$ given $\BadBlocks(k,\ell-1)=I$.
Hence, the above estimates give
\begin{align*}
\P\Bigl\{&|\BadBlocks(k,\ell)|\ge N\exp(-{C_h}^2 2^{(k+\ell)/2})\Bigr.\\
\Bigl.&\mbox{or }\|\BStat(k,\ell)\|>\frac{\sqrt{N}}{\exp({C_h}^2 2^{(k+\ell)/2})}\,\bigl|\bigr.\,\BadBlocks(k,\ell-1)=I\Bigr\}\\
&\le \P\bigl\{\BStat_i(k,\ell)>0\mbox{ for some }i\in I\,|\,\BadBlocks(k,\ell-1)=I\bigr\}+2\exp(-2\sqrt{N}).
\end{align*}
Now, summing over all admissible subsets $I$, we get
\begin{align*}
\P&\Bigl\{|\BadBlocks(k,\ell)|\ge N\exp(-{C_h}^2 2^{(k+\ell)/2})
\mbox{ or }\|\BStat(k,\ell)\|>\frac{\sqrt{N}}{\exp({C_h}^2 2^{(k+\ell)/2})}\Bigr\}\\
&\le 2\exp(-2\sqrt{N})\\
&\hspace{1cm}+\sum\limits_I \P\bigl\{\BStat_i(k,\ell)>0\mbox{ for some }i\in I\,|\,\BadBlocks(k,\ell-1)=I\bigr\}\P\{\BadBlocks(k,\ell-1)=I\}\\
&=2\exp(-2\sqrt{N})+\P\bigl\{\BStat_i(k,\ell)>0\mbox{ for some }i\in \BadBlocks(k,\ell-1)\bigr\}\\
&=2\exp(-2\sqrt{N})+1-\P(\Event_{k,\ell-1}).
\end{align*}
By analogous argument, as a corollary of \eqref{aux 167},
$$\P\bigl\{\BadBlocks(k,\ell)\neq\emptyset\bigr\}\le N\exp(-{C_h}^2\alpha^{-1})+1-\P(\Event_{k,\ell-1}).$$
\end{proof}

\begin{lemma}[Construction of $\bar\Delta_{k,\ell}$]\label{construction of deltakl}
Let $1\le k\le M$ and $1\le \ell\le M'$ and assume that the random unit vector
$\bar{n}_{k,\ell-1}$ satisfying properties \eqref{nkl prop 1} and \eqref{nkl prop 2}
has been constructed.
Then one can construct a random unit vector $\bar \Delta_{k,\ell}$ satisfying properties
\eqref{deltakl prop 1}---\eqref{deltakl prop 2}---\eqref{deltakl prop 3} and such that
$$\P(\Event_{k,\ell})\ge \P(\Event_{k,\ell-1})-3\exp(-\sqrt{N})$$
if $\ell>1$, or
$$\P(\Event_{k,\ell})\ge \P(\Event_{k-1})-3\exp(-\sqrt{N})$$
if $\ell=1$.
\end{lemma}
\begin{proof}
Fix for a moment any subset $I\subset\{0,1,\dots,N\}$ such that the event
$$\Event_I=\{\BadBlocks(k,\ell)=I\}$$
has a non-zero probability.
If $|I|>N\exp(-{C_h}^2 2^{(k+\ell)/2}/32)$ then define a ``random'' vector $\bar \Delta_{k,\ell}^I$
on $\Event_I$ by setting $\bar \Delta_{k,\ell}^I:=u$ for a fixed unit vector $u\in\R^{J^k_\ell}$.
Otherwise, if $|I|\le N\exp(-{C_h}^2 2^{(k+\ell)/2}/32)$, we proceed as follows:

For each $i\in I\setminus\{0\}$, define $2^k$ ``increments'' on $\Event_I$:
$$
X_{i,p}:= \frac{\Proj_\ell^k\bigl(\B_n(t_{i,p+1})- \B_n(t_{i,p})\bigr)}{\sqrt{t_{i,p+1}-t_{i,p}}},\;\;p=0,1,\dots,2^k-1,
$$
where $t_{i,p}= 2^{i-1+p2^{-k}}$ for $p=0,1,\dots,2^k$. Additionally, if $0\in I$, then define
$$X_{0,0}:=\Proj_\ell^k\B_n(1).$$
Let us denote by $T_I$ the set of all pairs of indices $(i,p)$ corresponding to the ``increments'' $X_{i,p}$.
Note that the process $\Proj^k_\ell\B_n(t)$ is independent from $\Event_I$; in particular,
$\{X_{i,p},\;(i,p)\in T_I\}$ is a collection of standard Gaussian vectors on $\Event_I$ with values in 
$\R^{J_\ell^k}$, such that all $X_{i,p}$ and the vector $\BStat(k,\ell)$ are 
{\it jointly independent} given $\Event_I$.
Let us define a random vector $\tilde b^I\in\R^{T_I}$ on $\Event_I$ by
$$\tilde b^I_{i,p}=\begin{cases}2^{-k/2}\BStat_i(k,\ell)/\|\BStat(k,\ell)\|,&\mbox{if }\BStat(k,\ell)\neq {\bf 0};\\
0,&\mbox{otherwise.}\end{cases}$$
It is easy to see that $\|\tilde b^I\|\le 1$ (deterministically) and that
$$|T_I|\le 2^k|I|\le 2^k N\exp(-{C_h}^2 2^{(k+\ell)/2}/32)\le \frac{1}{2}|J^k_\ell|.$$
(In the last estimate, we used the assumption that $C_h$ is a large constant).
Hence, in view of Lemma~\ref{normal-vector-lem}, there exists a random unit vector
$\bar\Delta_{k,\ell}^I\in\R^{J_\ell^k}$ on $\Event_I$ (which is a Borel function of $X_{i,p}$ and $\tilde b^I$) such that
\begin{align*}
\P\bigl\{\langle \bar\Delta_{k,\ell}^I,X_{i,p}\rangle
\ge c_{\ref{normal-vector-lem}}\sqrt{|J_\ell^k|}\tilde b^I_{i,p}\mbox{ for all }(i,p)\in T_I\,|\,\Event_I\bigr\}
&\ge 1-\exp(-c_{\ref{normal-vector-lem}}|J_\ell^k|)\\
&\ge 1-\exp(-\sqrt{N}).
\end{align*}
It will be convenient for us to denote by $\tilde \Event_I$ the event
$$\bigl\{\langle \bar\Delta_{k,\ell}^I,X_{i,p}\rangle
\ge c_{\ref{normal-vector-lem}}\sqrt{|J_\ell^k|}\,\tilde b^I_{i,p}\mbox{ for all }(i,p)\in T_I\bigr\}\subset\Event_I.$$

By ``glueing together'' $\bar\Delta_{k,\ell}^I$ for all $I$,
we obtain a random vector $\bar\Delta_{k,\ell}$ on the entire probability space.

Clearly, $\bar\Delta_{k,\ell}$ satisfies properties \eqref{deltakl prop 1} and \eqref{deltakl prop 2}.
Next, on each $\Event_I$ with $\P(\Event_I)>0$ the vector $\bar\Delta_{k,\ell}$ was defined
as a Borel function of $\BStat(k,\ell)$ and $\Proj^k_\ell(\B(t)-\B(\tau))$, $t,\tau\in I^i_k\cup\{a_i,a_{i+1}\}$, $i\in I$, so,
in view of the properties of the Brownian motion, $\bar\Delta_{k,\ell}$ satisfies \eqref{deltakl prop 3}.

Finally, we shall estimate the probability of $\Event_{k,\ell}$.  
Define
\begin{align*}
\Event=\Bigl\{|\BadBlocks(k,\ell)|\le N\exp(-{C_h}^2 2^{(k+\ell)/2}/32)\mbox{ and }
\|\BStat(k,\ell)\|\le\frac{\sqrt{N}}{\exp({C_h}^2 2^{(k+\ell)/2}/64)}\Bigr\}.
\end{align*}
Note that, according to Lemmas~\ref{BStat initial substep lemma} and~\ref{BStat subsequent substeps lemma},
the probability of $\Event$ can be estimated from below by $\P(\Event_{k,\ell-1})-2\exp(-2\sqrt{N})$
for $\ell>1$ and $\P(\Event_{k-1})-2\exp(-2\sqrt{N})$ for $\ell=1$.

Take any subset $I\subset\{0,1,\dots,N\}$ with $|I|\le N\exp(-{C_h}^2 2^{(k+\ell)/2}/32)$ and
such that $\tilde\Event_I\cap \Event\neq\emptyset$, and let $\omega\in\tilde \Event_I\cap\Event$.
If $\BadBlocks(k,\ell)=\emptyset$ at point $\omega$ then, obviously, $\omega\in\Event_{k,\ell}$.
Otherwise, we have
\begin{align*}
\bigl\langle &\bar\Delta_{k,\ell}(\omega),\frac{\B_n(t_{i,p+1})(\omega)- \B_n(t_{i,p})(\omega)}{\sqrt{t_{i,p+1}-t_{i,p}}}
\bigr\rangle\\
&\ge \frac{c_{\ref{normal-vector-lem}}2^{-k/2}\sqrt{|J_\ell^k|}\BStat_i(k,\ell)(\omega)}{\|\BStat(k,\ell)(\omega)\|}
\mbox{ for all }(i,p)\in T_I,
\end{align*}
whence, using the estimate $t_{i,p+1}-t_{i,p}\ge \frac{2^{i-k}}{4}$ ($(i,p)\in T_I$), we obtain for
any $i\in I$ and $t\in I^i_k\cup\{a_{i+1}\}$:
\begin{align*}
\langle &\bar\Delta_{k,\ell}(\omega),\B_n(t)(\omega)-\B_n(a_i)(\omega)\rangle\\
&=\sum_{p:\,t_{i,p}<t}  \langle \bar\Delta_{k,\ell}(\omega),\B_n(t_{i,p+1})(\omega)-\B_n(t_{i,p})(\omega)\rangle\\
&\ge \frac{c_{\ref{normal-vector-lem}}2^{-k-1}\sqrt{a_{i+1}|J_\ell^k|}\BStat_i(k,\ell)(\omega)}{\|\BStat(k,\ell)(\omega)\|}.
\end{align*}
Further,
$$\frac{c_{\ref{normal-vector-lem}}2^{-k-1}\sqrt{|J_\ell^k|}}{\|\BStat(k,\ell)(\omega)\|}
\ge \frac{c_{\ref{normal-vector-lem}}2^{-k-1}\sqrt{c_J n 2^{(-k-\ell)/8}}\exp({C_h}^2 2^{(k+\ell)/2}/64)}{\sqrt{N}}\ge \frac{1}{\alpha_{k,\ell}}.
$$
Using the definition of $\bar n_{k,\ell}$ in terms of $\bar n_{k,\ell-1}$ and $\bar\Delta_{k,\ell}$ and the above estimates,
we get
\begin{align*}
\langle &\bar n_{k,\ell}(\omega),\frac{\B_n(t)(\omega)-\B_n(a_i)(\omega)}{\sqrt{a_i}}\rangle\\
&\ge \frac{-h(k,\ell)-\BStat_i(k,\ell)(\omega)}{\sqrt{1+{\alpha_{k,\ell}}^2}}+
\frac{\alpha_{k,\ell}}{\sqrt{1+{\alpha_{k,\ell}}^2}}\langle \bar \Delta_{k,\ell}(\omega),
\frac{\B_n(t)(\omega)-\B_n(a_i)(\omega)}{\sqrt{a_i}}\rangle\\
&\ge\frac{-h(k,\ell)}{\sqrt{1+{\alpha_{k,\ell}}^2}}\\
&\ge -h(k,\ell+1),\;\;t\in I^i_k,\;\;i\in I,
\end{align*}
and, similarly,
$$\langle \bar n_{k,\ell}(\omega),\frac{\B_n(a_{i+1})(\omega)-\B_n(a_i)(\omega)}{\sqrt{a_{i+1}}}\rangle
\ge \frac{f(k,\ell)}{{\sqrt{1+{\alpha_{k,\ell}}^2}}}\ge f(k,\ell+1),\;\;i\in I.$$
Thus, by the definition of the event $\Event_{k,\ell}$, we get $\omega\in\Event_{k,\ell}$.

The above argument shows that
$$\P(\Event_{k,\ell})\ge\sum\limits_I\P(\tilde \Event_I\cap\Event),$$
where the sum is taken over all $I$ with $|I|\le N\exp(-{C_h}^2 2^{(k+\ell)/2}/32)$.
Finally,
$$\sum\limits_I\P(\tilde \Event_I\cap\Event)\ge\sum\limits_{I}\P(\Event_I\cap\Event)
-\sum\limits_I\P(\Event_I\setminus\tilde\Event_I)\ge\P(\Event)-\exp(-\sqrt{N}),$$
and we get the result.
\end{proof}

\begin{lemma}[$k$-th Step]\label{k step lemma}
Let $1\le k\le M$ and assume that a random unit vector $\bar n_{k-1}$ satisfying properties \eqref{nk prop 1}, \eqref{nk prop 2}
has been constructed. Then there exists a random unit vector $\bar n_{k}$ satisfying \eqref{nk prop 1}---\eqref{nk prop 2}
and such that
$$\P(\Event_k)\ge\P(\Event_{k-1})-\frac{1}{n^2}.$$
\end{lemma}
\begin{proof}
As before, we set $\bar n_{k,0}:=\bar n_{k-1}$. Consecutively applying Lemma~\ref{construction of deltakl}
and formula \eqref{def of nkl via delta} $M'$ times, we obtain a random unit vector
$\bar n_{k,M'}$ satisfying \eqref{nkl prop 1} and \eqref{nkl prop 2}.
Moreover, the same lemma provides the estimate
$$\P(\Event_{k,\M'})\ge\P(\Event_{k-1})-3M'\exp(-\sqrt{N}).$$
Then, in view of Lemma~\ref{BStat subsequent substeps lemma} and the definition of $M'$, we have
$$\P\bigl\{\BadBlocks(k,M'+1)\neq\emptyset\bigr\}\le N\exp(-{C_h}^2/\alpha_{k,M'})+1-\P(\Event_{k,M'})\le
\frac{1}{n^2}+1-\P(\Event_{k-1}).$$
Combining the above estimate with the definition of $\Event_k$, we get for $\bar n_k:=\bar n_{k,M'}$:
$$\P(\Event_k)\ge \P(\Event_{k-1})-\frac{1}{n^2}.$$
\end{proof}

\begin{proof}[Proof of Theorem~\ref{main result}]
Define a vector $b=(b_0,b_1,\dots,b_N)$ by
$$b_i:=2{c_{\ref{normal-vector-lem}}}^{-1}f(1,0),\;\;i=0,1,\dots,N.$$
In view of the definition of $f$ and the relation \eqref{relation between constants},
we have $\|b\|\le \sqrt{|J^0_1|}$, and, as we have chosen $c$ to be small, $N+1\le |J^0_1|/2$.
Hence, in view of Lemma~\ref{normal-vector-lem}, there exists a random unit vector $\bar n_0\in\R^{J^0_1}$ measurable
with respect to the $\sigma$-algebra generated by vectors $\Proj^0_1(\B_n(a_{i+1})-\B_n(a_i))$, $i=0,1,\dots,N$,
and such that
\begin{align*}
\P(\Event_0)&=
\P\bigl\{\langle \bar n_0,\B_n(a_{i+1})-\B_n(a_i)\rangle \ge f(1,0)\sqrt{a_{i+1}}\mbox{ for all }
i=0,1,\dots,N\bigr\}\\
&\ge\P\Bigl\{\bigl\langle \bar n_0,\frac{\B_n(a_{i+1})-\B_n(a_i)}{\sqrt{a_{i+1}-a_i}}\bigr\rangle \ge 
c_{\ref{normal-vector-lem}}b_i\mbox{ for all }
i=0,1,\dots,N\Bigr\}\\
&\ge 1-\exp(-c_{\ref{normal-vector-lem}}|J^0_1|)\\
&\ge 1-\frac{1}{n^2}.
\end{align*}
Applying Lemma~\ref{k step lemma} $M$ times, we obtain a random unit vector $\bar n_M$ satisfying
\eqref{nk prop 1}-\eqref{nk prop 2} such that
$$\P(\Event_M)\ge 1-\frac{M+1}{n^2}.$$
Note that for any $\omega\in\Event_M$, we have
$$\langle \bar n_M,\B_n(a_{i+1})(\omega)\rangle\ge \langle \bar n_M,\B_n(a_{i+1})(\omega)-\B_n(a_{i})(\omega)\rangle\ge C_f\sqrt{a_{i+1}}$$
and
$$\langle \bar n_M,\B_n(t)(\omega)\rangle \ge \langle \bar n_M,\B_n(a_i)(\omega)\rangle-\frac{C_f}{2}\sqrt{a_i}\ge\frac{C_f}{2}\sqrt{a_i},\;\;
t\in I^i_k$$
for all $i=0,1,\dots,N$.
Hence, denoting $Q:=\{a_1,a_2,\dots,a_{N+1}\}\cup\bigcup_{i=1}^N I^i_k$, we get
\begin{equation}\label{eventM is subset}
\Event_M\subset\Bigl\{\bigl\langle\bar n_M,\frac{\B_n(t)}{\sqrt{t}}\bigr\rangle\ge \frac{C_f}{4},\;\;t\in Q\Bigr\}.
\end{equation}
Now, take any two adjacent (i.e.\ neighbour) points $t_1<t_2$ from $Q$.
Note that, conditioned on a realization of vectors $\B_n(t)$, $t\in Q$, the random process
$$X(s)=
\bigl\langle\bar n_M,\frac{s\B_n(t_2)+(1-s)\B_n(t_1)}{\sqrt{t_2-t_1}}\bigr\rangle
-\bigl\langle \bar n_M,\frac{\B_n(t_1+s(t_2-t_1))}{\sqrt{t_2-t_1}}\bigr\rangle,\;s\in[0,1],$$
is the standard Brownian bridge. Hence (see, for example, \cite[p.~34]{SW}), we have for any $\tau>0$
$$\P\bigl\{X(s)\ge \tau\mbox{ for some }s\in [0,1]\bigr\}=\exp(-2\tau^2).$$
Taking $\tau:=2\sqrt{\ln n}$, we obtain
\begin{align*}
\P\bigl\{\bigl\langle\bar n_M,\B_n(t)\bigr\rangle
&\le\max\bigl(\langle\bar n_M,\B_n(t_1)\rangle,\langle\bar n_M,\B_n(t_2)\rangle\bigr)\bigr.\\
&\bigl.-2\sqrt{t_2-t_1}\sqrt{\ln n}\mbox{ for some }t\in[t_1,t_2]\bigr\}\\
&\hspace{-2cm}\le\frac{1}{n^8}.
\end{align*}
Finally, note that, in view of \eqref{eventM is subset}, everywhere on $\Event_M$ we have
\begin{align*}
&(t_2-t_1)^{-1/2}\max\bigl(\langle\bar n_M,\B_n(t_1)\rangle,\langle\bar n_M,\B_n(t_2)\rangle\bigr)-2\sqrt{\ln n}\\
&\ge \frac{C_f}{4}\sqrt{\frac{t_2}{t_2-t_1}}-2\sqrt{\ln n}\\
&\ge 2^{M/2-2}C_f-2\sqrt{\ln n}\\
&>0.
\end{align*}
Taking the union bound over all adjacent pairs in $Q$ (clearly, $|Q|\le n^2$), we come to the relation
$$\P\bigl\{\langle\bar n_M,\B_n(t)\rangle>0\mbox{ for all }t\in[1,2^{cn}]\bigr\}\ge\P(\Event_M)-\frac{|Q|}{n^8}\ge 1-\frac{1}{n}.$$
\end{proof}

\end{document}